\documentclass{article}
\usepackage{amsmath}
\usepackage{amssymb}
\usepackage{amsthm, amscd}
\usepackage[usenames]{color}

\newtheorem{thm}{Theorem}[section]
\newtheorem{prop}[thm]{Proposition}
\newtheorem{lem}[thm]{Lemma}
\newtheorem{cor}[thm]{Corollary}
\newtheorem{defn}[thm]{Definition}

\theoremstyle{remark}
\newtheorem{rem}[thm]{Remark}

\newtheorem{exmp}[thm]{Example}

\numberwithin{equation}{section}

\newcommand{\mbb}{\mathbb}
\newcommand{\fra}{\mathfrak}

\newcommand{\ov}{\overline}

\newcommand{\N}{\mathbb{N}}
\newcommand{\Z}{\mbb{Z}}
\newcommand{\Q}{\mbb{Q}}
\newcommand{\R}{\mbb{R}}

\newcommand{\ds}{\displaystyle}

{}

\begin{document}

\title{$\omega$-stability and Morley rank of bilinear maps, rings and nilpotent groups}

\author{Alexei G. Myasnikov\footnote{Department of Mathematical Sciences,
Stevens Institute of Technology,
Castle Point on Hudson,
Hoboken, NJ 07030, USA, email: amiasnikov@gmail.com}, Mahmood Sohrabi\footnote{Department of Mathematical Sciences,
  Stevens Institute of Technology,
 Castle Point on Hudson,
 Hoboken, NJ 07030, USA, email: msohrab1@stevens.edu}}

\maketitle

\begin{abstract}
In this  paper we study the algebraic structure of $\omega$-stable bilinear maps, arbitrary rings and nilpotent groups. We will also provide rather complete structure theorems for the above structures in the finite Morley rank case.  
\end{abstract}

Keywords: Bilinear map, Largest ring of scalars, Ring, Nilpotent group, $\omega$-stability, Morley rank

\section{Introduction}\label{se:intro}
In this  paper we provide structure theorems for the algebraic structure of $\omega$-stable bilinear maps, arbitrary rings and nilpotent groups. We will also provide rather complete structure theorems for the above structures in the finite Morley rank case.  Our results are based on several classical works from  1970's. Macintyre in \cite{mac1,mac2} showed that torsion-free $\omega$-stable abelian groups  are precisely the divisible ones, while infinite $\omega$-stable integral domains with unit are algebraically closed fields.  In \cite{CR} Cherlin and Reineke described algebraic structure of $\omega$-stable, as well as of finite Morley rank, commutative associative rings with unit. These are the results we use through out the paper. In \cite{Z82} Zilber gave a precise description of $\aleph_1$-categorical torsion-free nilpotent groups and rings of characteristic zero. Our description of groups and rings in question may be seen as a natural generalization of the Zilber's one, only taken further to the finite Morley rank and $\omega$-stable cases. However, our techniques are quite different, for example, we do not use Zilber's indecomposability theorem in our theorems on the Morley rank of rings and nilpotent groups. Instead, we rely on first-order interpretability of the largest ring of scalars and its actions in a bilinear map from \cite{M89,M90a,M90b}. This gives a powerful  general method in model-theoretical algebra that works for various groups and rings, which are not necessarily even  $\omega$-stable. Indeed,  the technique of bilinear maps was essential in proving the following results: description of the algebraic structure of models of complete theories of finite dimensional algebras (see Theorem~\ref{models-alg:thm} below for the exact statement in the directly indecomposable case) and unipotent groups, as well as in giving a decidability criterion of such theories \cite{M90c};  a precise description of arbitrary groups elementarily equivalent to a given free nilpotent group of finite rank \cite{MS2010}; Belegradek's characterization of groups elementarily equivalent to a given unitriangular group  $UT(n,\mathbb{Z})$; elementary coordinatization of finitely generated nilpotent groups  \cite{MS2014}. 
Recently, some interest in torsion-free nilpotent groups of finite Morley rank reappeared, though from a different angle. In \cite{AW} Altinel and Wilson studied algebraic structure of such groups proving, in particular that they have faithful linear representations over algebraically closed fields of characteristic zero. Their techniques are inspired by those of B. Zilber~\cite{Z74}. Some other remarkable results have been obtained by O. Fr\'econ~\cite{F06, F08} where the main objective is to find and study a suitable analogue of unipotence in the context of groups of finitely Morley rank. We encourage the reader to compare Fact 3.13 in ~\cite{F08}, which offers a structure theorem for nilpotent groups of finite Morley rank with our Theorem~\ref{fmr:garbit:mainthm} below. It seems that our result clarifies the precise structure of the components in the mentioned structure theorem.

We now describe the main results and structure of this paper. Section~\ref{classical:sec} collects some of the classical structure theorems for $\omega$-stable commutative rings and groups. 

We discuss bilinear maps in section~\ref{bilin:sec}. Most results in the torsion-free case are taken from~\cite{M90a} and~\cite{M90b} and inevitably we need to recall some material from those papers. Later in that section we provide rather complete structure theorems for arbitrary $\omega$-stable (of finite Morley rank) bilinear maps. 

In Section~\ref{algebras:sec} we analyze $\omega$-stable rings (of finite  Morley rank). For us a \emph{ring} is just an abelian group with operation $``+"$ together with a binary operation $``\cdot"$, where $``\cdot"$ distributes over $``+"$. However a \emph{scalar ring} is always assumed to be commutative associative with a unit. We say that a ring $R$ has \emph{characteristic zero} if its additive group $R^+$ is torsion-free.  Here are the main results in characteristic zero case.   
\begin{thm}\label{w-st-al:mainthm} Let $R$ be an $\omega$-stable ring of characteristic zero. Then $R$ can be decomposed into a finite direct product $$R\cong R_1\times\cdots \times R_n\times R_0$$ where each $R_i$ is an indecomposable $k_i$-algebra, $k_i$ is a characteristic zero algebraically closed field, and $R_0$ is a $\Q$-algebra with zero multiplication.\end{thm}
\begin{thm} \label{fmr-alg:mainthm}  Let $R$ be  a  ring of characteristic zero.  Then $R$ is of finite Morley rank if and only if $R$ can be decomposed into a finite direct sum $$R\cong R_1\times\cdots \times R_n\times R_0$$ where each $R_i$ is a finite dimensional indecomposable $k_i$-algebra, $k_i$ is a characteristic zero algebraically closed field and $R_0$ is a $\Q$-algebra with zero multiplication.\end{thm}

\begin{rem} Note that in Theorem~\ref{fmr-alg:mainthm} we did not claim that any of the components $R_i$, $i=0,\ldots, n$ are definable in $R$, and in general there is no reason that they are. However the statement implies that each of these components has finite Morley rank.\end{rem}
Proofs of Theorem~\ref{w-st-al:mainthm} and Theorem~\ref{fmr-alg:mainthm} will appear in Sections \ref{wstable-al:sec} and \ref{fmr-alg:sec} respectively. The statements regarding the arbitrary $\omega$-stable rings (of finite Morley rank) get more complicated and a bit less illuminating, but still rather complete. We do not state them here in the introduction and refer the reader to Section~\ref{algebras:sec}. We also provide a criterion for uncountable categoricity of rings in Section~\ref{algebras:sec} (Theorem~\ref{cat:alg:thm}). In Section~\ref{malcev-cor:sec} we give a quick overview of the so-called Mal'cev correspondence between rational Lie algebras and torsion-free divisible nilpotent groups. The reader may also refer to~\cite{AW} for further details. Finally in Section~\ref{groups:sec} we look at $\omega$-stable nilpotent groups, as well as at nilpotent groups of finite  Morley rank. The structure theorems in the torsion-free case will almost immediately follow from those on characteristic zero rings and the Mal'cev correspondence mentioned above. Here are the main results in this case. Proofs will appear at the beginning of Section~\ref{groups:sec} and Section~\ref{fmrG:section}.     

\begin{thm}\label{w-stable-group:mainthm} Assume $G$ is a torsion-free $\omega$-stable nilpotent group. Then $G$ is a finite direct product $$G\cong G_1\times \cdots \times G_n\times G_0$$ where each $G_i$, $i\neq 0$, is an indecomposable $k_i$-group where $k_i$ is a characteristic zero algebraically closed field and $G_0$ is torsion-free divisible abelian group.\end{thm}

\begin{thm} \label{fmr:groups:mainthm} A torsion-free nilpotent group $G$ has finite Morley rank in the language of groups if and only if $G$ is a finite direct product $$G\cong G_1\times \cdots \times G_n\times G_0$$ where each $G_i$, $i\neq 0$, is a unipotent algebraic group over $k_i$, $k_i$ is a characteristic zero algebraically closed field, and $G_0$ is a torsion-free divisible abelian group.\end{thm}

Then applying the following theorems of A. Nesin and our results we get the structure theorems for arbitrary $\omega$-stable nilpotent groups (of finite Morley rank).

We will state A. Nesin's theorems below right after the following definition.

A group $G$ is said to be a \emph{central product} of some of its subgroups $H_1, \ldots, H_n$ if $H=H_1\cdots H_n$, $H_i\unlhd G$ for all $i=1, \ldots n$, and $[H_i,H_j]=1$ if $i \neq j$. Then we write $G=H_1\ast
 \cdots \ast H_n.$ Recall that for subgroups $H$ and $K$ of $G$, $[H,K]$ is the subgroup generated by all the commutators $x^{-1}y^{-1}xy=[x,y]$, $x\in H$, $y\in K$.

 \begin{thm}[\cite{N91}, Theorem~2]\label{nesin1} Let $G$ be an $\omega$-stable nilpotent group. Then $G$ is a central product $D\ast C$ where $D$ and $C$ are definable characteristic subgroups of $G$, $D$ is divisible, and $C$ has bounded exponent.   
 \end{thm}
 \begin{thm}[\cite{N91}, Corollary to Theorem~2]\label{nesin1a}  The nilpotent groups of finite Morley rank are exactly the central products $D\ast C$ where $D$ and $C$ both have finite Morley rank, $D$ is divisible,  $C$ has bounded exponent, and $D\cap C$ is finite. 
 	\end{thm}
 \begin{thm}[\cite{N91}, Lemma 3 and Theorem 3]\label{nesin2}
 Let $G$ be a divisible $\omega$-stable nilpotent group. Then $G'=[G,G]$ is torsion-free. Moreover if $T$ denotes the torsion part of $G$, then $T$ is central in $G$ and $G=T\times N$ for some torsion-free divisible nilpotent subgroup $N$. If $G$ has finite Morley rank $T$ is also of finite Morley rank.\end{thm}
 
 Finally here are our structure theorems for $\omega$-stable nilpotent groups (of finite Morley rank).  
  
 \begin{thm} \label{wstable:arbitgroup:mainthm}Let $G$ be an $\omega$-stable nilpotent group. Then $G$ is a central product $D\ast C$ where $D$ and $C$ are definable characteristic subgroups of $G$, $D$ is divisible, and $C$ has bounded exponent. Moreover $D$ has a direct decomposition 
 $$D=D_1\times \cdots \times D_n \times D_0$$ where each $D_i$, $i\neq 0$, is a directly indecomposable $k_i$-group  over an algebraically closed field $k_i$ of characteristic zero and $D_0$ is a  divisible abelian group.
 \end{thm} 
\begin{thm}\label{fmr:garbit:mainthm} Let $G$ be a nilpotent group. Then the following statements are equivalent.
\begin{enumerate}
	\item $G$ has finite Morley rank.
	\item $G$ is a central product $D\ast C$ where $D$ and $C$ are definable characteristic subgroups of $G$, $D$ is divisible, $C$ has bounded exponent and $D\cap C$ is finite. Both $G$ and $D$ have finite Morley rank. Moreover $D$ has a direct decomposition 
	$$D=D_1\times \cdots \times D_n \times D_0$$ where each $D_i$, $i\neq 0$, is a directly indecomposable unipotent algebraic group over an algebraically closed field $k_i$ of characteristic zero and $D_0$ is a divisible abelian group of finite Morley rank. 
\end{enumerate}  
 \end{thm}
   
From our results we can also provide a positive solution to Conjecture 6.10  of A. Borovik and A. Nesin~\cite{BN94}.

\begin{cor}[]\label{BN-conjecture} The subgroup $N$ in Theorem~\ref{nesin2} has finite  Morley rank, provided the group $G$ has finite Morley rank.\end{cor}

 We would like to remark here that the structure theorems for arbitrary $\omega$-stable bilinear maps and rings are similar to Theorem~\ref{wstable:arbitgroup:mainthm} and Theorem~\ref{fmr:garbit:mainthm}. In those cases we shall also give a better description of the bounded exponent component (See Proposition~\ref{boundedstablebilin:prop} and Proposition~\ref{wstable:boundedalg:prop}).
 
  Finally we would like to mention an example due to A. Baudisch. In~\cite{baud}, he constructs an uncountable categorical nilpotent group of exponent $p$, for a prime $p>2$, which does not interpret any infinite fields. In particular it shows that Theorems~\ref{w-stable-group:mainthm} and \ref{fmr:groups:mainthm} can not be extended to the bounded exponent case. 

\section{Some classical results on $\omega$-stable algebraic structures.}
\label{classical:sec}

Here we collect a few classical results which will be frequently used in the sequel.

\begin{thm}[A. Macintyre~\cite{mac1}]\label{Mac:stable:tfa} An abelian group $M$ is $\omega$-stable if and only if $M= M_D\oplus M_B$ where $M_D$ is divisible and $M_B$ is of bounded order. $M_D$ is characteristic and absolutely definable in $M$. The subgroup $M_B$ can be replaced by a characteristic absolutely definable subgroup $M_C$, but then we only have $M=M_D + M_C$. If $M$ has finite Morley rank then $|M_D\cap M_C|<\infty$.\end{thm}
\begin{thm}[A. Macintyre~\cite{mac2}] An infinite $\omega$-stable integral domain with unit is an algebraically closed field.\end{thm} 
 \begin{thm}[G. I. Cherlin and J. Reineke~\cite{CR}]\label{CR:stable} Let $R$ be a commutative associative $\omega$-stable unitary ring. Then $R$ is decomposed into a finite direct product $R=R_1\times\cdots \times R_n$ of local rings $R_i$, where the maximal ideal $J_i$ of $R_i$ is nilpotent and the field of residues $R_i/J_i$ is either algebraically closed or finite.\end{thm}

\begin{thm}[G. I. Cherlin and J. Reineke~\cite{CR}] \label{CR:fmr} Let $R$ be a commutative associative unitary ring. 
\begin{enumerate}
\item $R$ is of finite Morley rank if and only if it can be decomposed into a finite direct product $R=R_1\times\cdots \times R_n$ of local rings $R_i$, where each $R_i$ is definable in $R$ and of finite Morley rank, the maximal ideal $J_i$ of $R_i$ is nilpotent, and the field of residues $R_i/J_i$ is algebraically closed or finite. 
\item If $R$ is local and of finite Morley rank with algebraically closed field of residues then $R$ is Noetherian.   
\end{enumerate}
\end{thm}

\section{Bilinear maps}\label{bilin:sec}
Here we assume that $R$ is a commutative associative unitary ring and that $M_1$, $M_2$ and $N$ are $R$-modules. For the most part we restrict ourselves to those $R$-bilinear maps $f:M_1\times M_2 \to N$ where $M_1=M_2$. For our purposes this case is general enough, even though in all we do the general case $M_1\neq M_2$ reduces to the special case (See \cite{M90b}, Section 2.1)
 
 So Assume $f:M \times M \to N$ is an $R$-bilinear map. The bilinear map $f$ is called \emph{full} if the corresponding canonical $R$-homomorphism $\bar{f}: M\otimes_R M \to N$ is surjective.
 
 Assume $M_1$ and $M_2$ are subsets of $M$, then by $\langle f(M_1,M_2)\rangle$ we denote the $R$-submodule of $N$ generated by $f(x,y)$, $x\in M_1$, $y\in M_2$. In particular by $im(f)$ we mean $\langle f(M,M) \rangle$. 
 
The \emph{two sided kernel of $f$}, $C(f)$, is the $R$-submodule of $M$ defined by
 $C(f)=\{x\in M: f(x,M)=f(M,x)=0\}$. A mapping $f$ is called \emph{identically degenerate} if $f(x,y)=0$ for all $x,y\in M$. The map $f$ is called \emph{non-degenerate} if $C(f)=0$. To any bilinear map $f$ one can associate a full non-degenerate map 
 $$f^F:M/C(f)\times M/C(f)\to im(f)$$ in the obvious way. We call $f^F$ \emph{the foundation of $f$}. In case there is an $R$-submodule $M_F$ of $M$ such that $M=M_F\oplus C(f)$ we call the full non-degenerate map $f^F:M_F\times M_F\to im(f)$ induced by $f$ \emph{a foundation of $f$} for obvious reasons.  

\subsection{Maximal scalar ring $P(f)$ of a bilinear map $f$}\label{bilin-scalar:sec}
 In this section we define a ``canonical'' scalar ring $P(f)$ for a bilinear map $f$. The construction is taken from~\cite{M90a, M90b} and we refer the reader to those references for the details left out here. To follow the mentioned references we call it  \emph{largest scalar ring of $f$}. Even though it is the largest scalar ring with respect to which $f$ remains bilinear (assuming certain conditions on $f$) in a proper sense we do not use the maximality in this paper. As it was mentioned in the introduction (and will be observed in detail below) this commutative object reflects many of the logical and algebraic aspects of the highly non-commutative or non-symmetric underlying object. 
 
 All the modules are considered to be faithful. Scalar rings are always commutative associative with a unit.

Now assume $f:M\times M \to N$ is a non-degenerate
full $R$-bilinear mapping. We need to introduce a few auxiliary objects first.

An $R$-endomorphism $A$ of the $R$-module $M$ is called \textit{symmetric} if $$f(Ax,y)=f(x,Ay)$$ for every $x,y
\in M$. The set of all these endomorphisms is denoted by $Sym_f(M)$.
Set $$Z(f)=\{B\in Sym_f(M):A\circ B= B\circ
A,\quad \forall A\in Sym_f(M)\}.$$ Then $Z(f)$ is non-empty since the unit 1 belongs to $Z(f)$ and it is actually an $R$-subalgebra of $R$-endomorphisms $End_R(M)$ of $M$. For each $n$, let $Z_n(f)$
be the set of all endomorphisms $A$ in $Z(f)$ that satisfy the formula
\begin{equation}\label{Z(f)}\begin{split}
S_n(A)\Leftrightarrow & \forall \bar{x},\bar{y},\bar{u},\bar{v}
(\sum_{i=1}^nf(x_i,y_i)=\sum_{i=1}^{n}f(u_i,v_i) \rightarrow \\
& \sum_{i=1}^nf(Ax_i,y_i)=\sum_{i=1}^{n}f(Au_i,v_i)).
\end{split}\end{equation}
i.e.
\begin{equation}
Z_n(f)=\{A\in Z(f): S_n(A)\}.\label{zn}
\end{equation}
Each $Z_n(f)$ is also  an $R$-subalgebra of $Z(f)$. Now define $$P(f)=_{\text{def}}\cap_{i=1}^\infty Z_n(f).$$ The identity mapping is in
every $Z_n(f)$ so $P(f)$ is not empty. 

Note that since the mapping $f$ is full, for every $x\in N$ there are $x_i$ and
$y_i$ in $M$ such that $x=\sum_{i=1}^nf(x_i,y_i)$ for some $n$. The $P(f)$-module $M$ is faithful by construction. Now we can define the action
of $P(f)$ on $N$ by setting $Ax=\sum_{i=1 }^nf(Ax_i,y_i)$. The action is clearly well-defined since $A$
satisfies all the $S_n(A)$ and makes $N$ into an $P(f)$-module. Moreover for any $x,y\in M$ and $A\in P(f)$ we have $$f(Ax,y)=f(x,Ay)=Af(x,y),$$
that is, $f$ is $P(f)$-bilinear. Note that $P(f)$ is independent of the ring $R$.
\subsection{Decomposing bilinear maps}

\begin{defn}Given $f:M\times M \to N$ we say that $f$ is decomposed to a direct sum of $R$-bilinear mappings $f_1$ and $f_2$ and write $f=f_1\oplus f_2$ if there are $R$-submodules $M_1, M_2, M_3$ and $M_4$ of $M$ and $N_1$ and $N_2$ of $N$ such that
\begin{enumerate}
\item $M=M_1\oplus M_2=M_3\oplus M_4$ and $N=N_1\oplus N_2$,
\item $f_1:M_1\times M_3\to N_1$ and  $f_2: M_2\times M_4 \to N_2$ are $R$-bilinear mappings,
\item $f(x,y)=f_1(x_1,y_1)+f_2(x_2,y_2)$ for all $x,y\in M$, where $x=x_1+x_2$, $y=y_1+y_2$, $x_1\in M_1$, $x_2\in M_2$, $y_1\in M_3$ and $y_2\in M_4$.
\end{enumerate} 
If condition (1.) is replaced with
\begin{enumerate}
\item[1.$^*$] $M=M_1+M_2=M_3+M_4$ and $N=N_1+N_2$, 
\end{enumerate}
but still (2.) and (3.) hold, then we write $f=f_1+f_2$. 
\end{defn} 

\begin{lem}[\cite{M90a}, Proposition 3.1] Given a non-degenerate full bilinear mapping $f:M\times M \to N$, $f$ admits a direct decomposition into directly indecomposable mappings $f=f_1\oplus \cdots \oplus f_n$ if and only if $P(f)=P_1\oplus \cdots \oplus P_n$, where each $P_i$ is a directly indecomposable subring of $P(f)$ and $P_i=P(f_i)$.\end{lem}

We will need the following lemma for further results. 
\begin{lem}\label{bilin:decomp:lem} Assume $f:M\times M \to N$ is an $R$-bilinear mapping where $R$ is a commutative associative unitary ring. Assume $M$ and $N$ split as $R$-modules over $C(f)$ and $im(f)$, respectively. Let $N^0$ denote an $R$-complement of $im(f)$ in $N$ and let $M_F$ denote an $R$-complement of $C(f)$ in $M$. Then $f$ decomposes into  $f=f^F\oplus f^0$, where $f^F: M_F\times M_F \to im(f)$ is a non-degenerate full $R$-bilinear mapping and $f^0:C(f)\times C(f)\to N^0$ is the identically degenerate bilinear mapping induced by $f$. \end{lem} 
\begin{proof} Let $f^F$ of  be the restriction of $f$ to $M_F\times M_F$. Then $f^F:M_F\times M_F\to im(f_F)$ is full and non-degenerate. Note that $im(f^F)=im(f)$ by definition. The restriction $f^0$ of $f$ to $C(f)\times C(f)$ is identically degenerate and in particular $im(f^0)=0\leq N^0$. Moreover for any $(x,y)\in M\times M$ there are unique $x_1,y_1\in M_F$, and $x_2,y_2\in C(f)$ such that $x=x_1+x_2$, $y=y_1+y_2$ and 
\begin{align*}f(x,y)&=f(x_1+x_2,y_1+y_2)\\
&=f(x_1,y_1)+f(x_1,y_2)+f(x_2,y_1)+f(x_2,y_2)\\
&=f(x_1,y_1)+f(x_2,y_2)\\
&=f^F(x_1,y_1)+f^0(x_2,y_2)\end{align*}
This proves the statement.\end{proof} 

\begin{defn}[Foundations and additions] The bilinear maps $f^F$ and $f^0$ obtained from $f$ in Lemma~\ref{bilin:decomp:lem} are called a foundation and an addition of $f$, respectively.\end{defn}  

\subsection{Bilinear maps as multi-sorted structures}
Given an $R$-bilinear map $f:M\times M \to N$, we associate two
multi-sorted structures to it. One $$\mathfrak{U}_R(f)=\langle R,M,N,\delta,s_M,s_N\rangle$$ where
the predicate $\delta$ describes $f$ and $s_M$ and $s_N$ describe the actions of $R$ on the modules $M$
and $N$ respectively. The other one
$$\mathfrak{U}(f)=\langle M,N, \delta\rangle$$
which contains only a
predicate $\delta$ describing the mapping $f$. When we say \emph{$R$ and its actions on $M$ and $N$ are interpretable in $f$} we mean that $\fra{U}_R(f)$ is interpretable in $\fra{U}(f)$.

\subsection{A structure theorem for $\omega$-stable bilinear maps} \label{st:bilin:sec}
This section contains a rather complete analysis of the structure of $\omega$-stable bilinear maps. Indeed we do not give a single structure theorem but rather a collection of statements that can not be realistically assembled in one theorem. 
    
Recall the decompositions $M=M_D\oplus M_B=M_D+M_C$ from Theorem~\ref{Mac:stable:tfa} for an $\omega$-stable abelian group $M$. Also note that even though $M_D$ is absolutely definable in $M$ the subgroup $M_B$ is just a complement of $M_D$ in $M$ and not necessarily definable (See~\cite{BN94}, Exercise 7, page 78). However $M_B\cong M/M_D$, and as such $M_B$ is absolutely interpretable in $M$. Moreover $M_D=M_{\Q}\oplus M_{T}$ where $M_{\Q}$ is a $\Q$-vector space and $M_{T}$ is a torsion divisible group. This later decomposition is not necessarily a definable decomposition either. On the other hand $M_C$ is absolutely definable in 
$M$ but $M_C\cap M_D$ is not necessarily trivial. Lastly, we clearly have $M_C\cap M_D< M_T$. We will follow this notation for the following results. 

\begin{prop} \label{InDec:stablebilin:prop} Let $f:M\times M \to N$ be an $\omega$-stable bilinear map. Then $f=f_D+f_C$ where $f_D:M_D\times M_D \to N_D$ in a bilinear map of divisible abelian groups and $f_C:M_C\times M_C\to N_C$ is a bilinear map of abelian groups of bounded exponent. Both $f_D$ and $f_C$ are absolutely definable in $f$. As such both $f_D$ and $f_C$ are $\omega$-stable bilinear maps. Moreover $M_D\cap M_C\leq M_T\leq C(f)$.\end{prop}
\begin{proof}
Firstly, by definition $M$ and $N$ are definable in $f$. So they are both $\omega$-stable. Therefore $M=M_D+M_C$ and $N=N_D+N_C$. Now pick arbitrary $x\in M_D$ and $y\in M_C$. By assumption there is $m\in \N^+$ such that $my=0$. Moreover by divisibility of $M_D$ there is $x' \in M_D$ such that $mx'=x$. Therefore 
\begin{equation}
\label{tor-div:eqn}
f(x,y)=f(mx',y)=f(x',my)=0.
\end{equation} 
This shows that $f=f_D+f_C$ providing that $im(f_D)\leq N_D$ and $im(f_C)\leq N_C$, which are both easy to verify using bilinearity. The components $f_D$ and $f_C$ are both absolutely definable in $f$ since all the involved subgroups are so. Moreover $M_B\cong M/ M_D$ is absolutely interpretable in $M$ and therefore in $f$. The same holds for $N_B$. Therefore $f_B$ is absolutely interpretable in $f$. Hence both $f_D$ and $f_B$ are $\omega$-stable as claimed.

To prove the moreover part pick any $x\in M_T$ and any $t\in M$. There are $y\in M_D$ and $z\in M_C$ such that $t=x+y$. there is also $m\in \N^+$ such that $mx=mz=0$. Again there exist $x'\in M_T$ and $y'\in M_D$ such that $mx'=x$ and $my'=y$. Then
\begin{align*}
f(x,t)&=f(x,y+z)\\
& =f(x,y)+f(x,z)\\
&= f(mx,y')+f(x',mz)\\
&=0
\end{align*}
The argument can clearly be repeated with respect to the second variable. So $M_D\cap M_C\leq M_T\leq C(f)$, the first inequality being a clear one. \end{proof}

\begin{lem}[\cite{M90b}, Proposition 5.1]\label{P(f)-inter:lem} Let $f$ be an $\omega$-stable non-degenerate bilinear map. Then $P(f)$ and its action on $M$ are interpretable in $f$, i.e. the two-sorted module $\langle P(f), M, s_M\rangle$ is interpretable in $\fra{U}(f)$.\end{lem} 
We need to remark that the above result is stated in~\cite{M90b} for torsion-free bilinear maps. However a quick study of the proof shows that the statement holds for the arbitrary case. 
\begin{prop} \label{Div:stableBilin:prop} Assume $f:M_D\times M_D\to N_D$ is an $\omega$-stable bilinear map of divisible abelian groups. Then $f= f_{\Q}\oplus f^0$ where $f_{\Q}$ is a foundation of $f$ and $\Q$-bilinear, and $f^0$ is an addition of $f$.  \end{prop}
\begin{proof}
	To make notation of the proof less complicated we assume $f:M\times M\to N$ is an $\omega$-stable bilinear map of divisible groups. Firstly we prove that $C(f)$ is divisible. Assume $z\in C(f)$ and pick any non-zero integer $m$ and consider $z'\in M$ such that $mz'=z$. Then for any $x\in M$ there exists $x'\in M $ such that $mx'=x$ and therefore
	$$f(z',x)=f(z,x')=0=f(x',z)=f(x,z').$$ So $z'\in C(f)$. Since $C(f)$ is divisible there is a subgroup $M_F$ of $M$ where $M=M_F\oplus C(f)$. Next we prove $M_F$ is divisible and torsion-free, or rather that $M/C(f)$ is so, since $M_F\cong M/C(f)$. It is enough to prove that it is torsion-free since $M$ is divisible. If $mx\in C(f)$ for a nonzero integer $m$ then for any $y\in M$, there exists $y'\in M$ such that $my'=y$. Therefore $f(x,y)=f(mx,y')=0=f(y',mx)=f(y,x)$. So $x\in C(f)$. This proves torsion-freeness of $M_F$. Consider the largest scalar ring $P(f)$ of $f$ which is $\omega$-stable since by Lemma~\ref{P(f)-inter:lem} $P(f)$ and its action on $M$ are interpretable in $f$. $M_F$ is a $\Q$-algebra so is $P(f)$. By definition $P(f)$ has a well-defined action on $im(f)$. Since $P(f)$ is a $\Q$-algebra, $im(f)$ is also a $\Q$-algebra and therefore torsion-free. That $im(f)$ is divisible is a direct corollary of definitions. Since $im(f)$ is divisible it has a complement $N^0$ in $N$. So $f^0:C(f)\times C(f) \to N^0$ is an identically degenerate map, $f_\Q: M_F\times M_F \to im(f)$ induced by $f$ is the corresponding foundation and $\Q$-bilinear.  \end{proof}
	\begin{cor}\label{st:fullnodeg:cor} Assume $f:M\times M \to N$ is a full non-degenerate $\omega$-stable bilinear map. Then in the notation of Propositions~\ref{InDec:stablebilin:prop} and \ref{Div:stableBilin:prop}, $f$ admits a direct decomposition $f=f_\Q\oplus f_C$ where both $f_\Q$ and $f_C$ are absolutely definable in $f$ and therefore $\omega$-stable.\end{cor}
	\begin{proof}
	Firstly $M_T\leq C(f)=0$ by Proposition~\ref{InDec:stablebilin:prop} and the non-degeneracy assumption. So $M=M_\Q \oplus M_C$, where $M_\Q$ is absolutely definable in $M$ as $M_D$. Clearly now $f_D=f_\Q$. The hypothesis that $f$ is a full bilinear map implies that
	$$N=im(f)=im(f_\Q+f_C)=im(f_\Q)+im(f_C).$$
	By Proposition~\ref{Div:stableBilin:prop}, $im(f_\Q)$ is torsion-free, so $im(f_\Q)\cap im(f_C)=0$, and $N=im(f_\Q)\oplus im(f_C)$, while clearly $N_D=im(f_\Q)$, and $N_C=im(f_C)$. The rest of the statement should be clear.    
	\end{proof}
	 
\begin{prop}[\cite{M90b}, Theorem 5.2] \label{QDiv:stableBilin:prop} Assume $f$ is an $\omega$-stable bilinear map of torsion-free abelian groups. Then it decomposes into $$f=f_1 \oplus \cdots \oplus f_n \oplus f^0$$ where each $f_i$, $i=1,\ldots n$ is a directly indecomposable $k_i$-bilinear mapping for some algebraically closed field $k_i$ of characteristic zero and $f^0$ is an identically degenerate $\Q$-bilinear map.
\end{prop}

\begin{cor} \label{Div:stableBilin:cor} Assume $f$ is an $\omega$-stable bilinear map of divisible abelian groups. Then it decomposes into $$f=f_1 \oplus \cdots \oplus f_n \oplus f^0$$ where each $f_i$, $i=1,\ldots n$ is a directly indecomposable $k_i$-bilinear mapping for some algebraically closed field $k_i$ of characteristic zero and $f^0$ is an identically degenerate bilinear map of divisible abelian groups.
\end{cor}

\begin{proof} Clear from propositions~\ref{Div:stableBilin:prop} and \ref{QDiv:stableBilin:prop}.\end{proof} 
Before analyzing the bounded case we need to state a classical result on the structure of local complete rings due to I. S. Cohen. We say that a local ring $P$ with the unique maximal ideal $J$ is \emph{equicharacteristic} if $char(P)=char(P/J)$. We say that such a ring $P$, \emph{admits a field of representatives} if $P$ contains a subfield $L$ that maps onto $P/J$ under the canonical homomorphism $P\to P/J$.
\begin{thm}[I. S. Cohen (See~\cite{ZSII}, Theorem 27, Page 304)] \label{cohen:thm}A complete local equicharacteristic ring admits a field of representatives.\end{thm}
 
 \begin{lem}\label{nilrep:lem} Assume $P$ is a local ring whose maximal ideal $J$ is nilpotent. If $char(P)=0$ or $char(P)\neq 0$ and $P$ is equicharacteristic then $P$ admits a field of representatives.\end{lem}
\begin{proof} 	
	Firstly we prove that regardless of the characteristic of $P$,  it is isomorphic to its $J$-adic completion $\hat{P}_J$ and therefore it is complete in its $J$-adic topology. Recall that given the sequence of ideals $(J^i)_{i\in \N}$, elements of $\hat{P}_J$ are sequences $(x_i+J^i)_{i\in \N}\in \prod_{i\in \N}P/J^i$, where $x_i-x_j\in J^i$ for all $j\geq i$. If $J$ is nilpotent, say $J^n=0$ for some $n$, each such sequence is completely determined by $x_n\in P$. So the diagonal injection $\iota: P \to \hat{P}_J$, $x\mapsto (x+J^i)$ is indeed a surjection.
	
	Therefore $char(P)\neq 0$ case already follows from Theorem~\ref{cohen:thm} since $P$ is assumed to be equicharacteristic. So assume $char(P)=0$. We need to show that $P$ is equicharacteristic. Consider the copy of the ring of integers $\Z\cdot 1_P$ in $P$ also denoted by $\Z$. Note that $\Z\cap J=0$, otherwise $J^n\neq 0$ for any $n\in \N^+$. This means that $\Z$ embeds into the field $P/J$ and so $P$ is equicharacteristic.    \end{proof}  
\begin{prop}\label{boundedstablebilin:prop}
Assume $f:M_B\times M_B\to N_B$ is an $\omega$-stable bilinear map of abelian groups of bounded exponent. Then the full non-degenerate bilinear mapping $f^F:M_B/C(f)\times M_B/C(f)\to im(f)$ induced by $f$ decomposes into  
$$f^F=f^F_1\oplus \cdots \oplus f_n^F,$$
where each $f^F_i$, $i=1, \ldots, m$ is a directly indecomposable bilinear mapping over a ring $P_i$. 
Each ring $P_i$ is a local ring with a maximal nilpotent ideal $J_i$ and the field $F_i= P_i/J_i$ is either algebraically closed of positive characteristic or finite. In case that $P_i$ is equicharacteristic, $P_i$ contains fields of representatives and therefore $f^F_i$ is an $F_i$-bilinear map.  

\end{prop}
\begin{proof}
Firstly note that $C(f)$ is definable in $f$. So indeed $f':M_B/C(f_B)\times M_B/C(f)\to N$ is interpretable in $f$ and so $\omega$-stable. So to simplify the notation of this proof we assume that $f:M\times M \to N$ is an $\omega$-stable non-degenerate bilinear map of abelian groups of bounded exponent.  In particular $P(f)$ is an $\omega$-stable scalar ring. So by Theorem~\ref{CR:stable}, $P=P(f)$ admits a unique decomposition $P=P_1\times \cdots \times P_n$ into indecomposable ideals (each of which is a ring with a unit $e_i$). Each $P_i$ is a local ring with the nilpotent maximal ideal $J_i$ and the field $F_i=P_i/J_i$ is algebraically closed or finite. Here since $M$ is assumed to be of bounded order all fields have to be of finite characteristic.
 
 Define $M_i=e_iM=P_iM$ and let $f_i$ be the restriction of $f$ to $M_i\times M_i$. Note that by definition $P$ has a well-defined action on $im(f)$ even though $im(f)$ might not be definable in $N$. If we set $N_i=im(f_i)$ then $im(f)=N_1\oplus\cdots \oplus N_m$ and $f^F_i:M_i\times M_i\to N_i$ is $P_i$-bilinear. Therefore for $f^F:M\times M\to im (f)$ we have, 
 $$f^F \cong f^F_1 \oplus \cdots \oplus f^F_m,$$
 where $P(f_i)=P((f^F)_i)=P_i$.  Moreover by Lemma~\ref{nilrep:lem} if $char(A_i)=char(F_i)$ then $P_i$ contains a field of representatives of $F_i$. Hence, in that case $f^F_i$ will be an $F_i$-bilinear map. Otherwise in general $f^F_i$ is just $P_i$-bilinear.\end{proof}

\subsection{Bilinear maps of finite Morley rank}\label{fmr:bilin:sec}
Here we discuss bilinear maps of finite Morley rank. Our plan follows the $\omega$-stable case. Again regarding the subscripts $B$, $C$, $D$, $T$ and $\Q$ for abelian groups and bilinear maps we observe the convention we made at the beginning of Section~\ref{st:bilin:sec}.  

As a direct corollary of Proposition~\ref{InDec:stablebilin:prop} we get:
\begin{cor}\label{bilin:fmr:cor} Let $f:M\times M\to N$ be a bilinear mapping of finite Morley rank. Then $f= f_D + f_C$ where both $f_D$ and $f_C$ are definable in $f$ and consequently they are both of finite Morley rank. 
\end{cor}

 A full bilinear map $f:M\times M \to N$ is said to {\em have finite width} if there is a natural number $s$ so that for every $z\in N$, there are $(x_i,y_i)\in M\times M$, $i=1, \ldots , s$ such that
$$z=\sum_{i=1}^s f(x_i,y_i).$$ 
\begin{cor}\label{bilin:fwidth} If the bilinear map $f:M\times M\to N$ is of finite width then $im(f)$ is definable in $f$.\end{cor}
\begin{prop} \label{Div:fmrBilin:prop} Assume $f:M_D\times M_D\to N_D$ is a bilinear map of divisible abelian groups. Then the following are equivalent:   
\begin{enumerate}
\item $f$ is of finite Morley rank.
\item $f=f_{\Q}\oplus f^0$ where $f_{\Q}: M_{\Q}\times M_{\Q}\to N_{\Q}$ is a foundation of $f$ and $\Q$-bilinear, and $f^0$ is an addition of $f$. Both maps $f_\Q$ and $f^0$ have finite Morley rank. 
\end{enumerate}\end{prop}
\begin{proof} Proof of (1) $\Rightarrow$ (2) is similar to the proof of part (1) of Proposition~\ref{Div:stableBilin:prop}, except the statement regarding the Morley rank of $f_\Q$ and $f^0$, the proof of which will be rather indirect. Still similar to the proof of Proposition~\ref{Div:stableBilin:prop} $P(f)$ and its action on $M_\Q$ are interpretable in $M=M_D$. So $P=P(f)$ has finite Morley rank. Since $M_\Q$ is a $\Q$-vector space, $P(f)$ is a $\Q$-algebra. So by Theorem~\ref{CR:fmr} it admits a decomposition 
$$P=P_1\times \cdots \times P_n,$$ 
where each $P_i$, $i=1,\ldots n$, is an indecomposable Noetherian local ring with maximal ideal $J_i$ where $k_i=P_i/J_i$ is an algebraically closed field. Since $P$ is a $\Q$-algebra $k_i$ is of characteristic zero. So by Lemma~\ref{nilrep:lem} and Noetherity of the $P_i$, each of them is a finite dimensional $k_i$-algebra. This immediately proves that $f_\Q$ which is $P(f)$-bilinear is a bilinear map of finite width and therefore $im(f)=im(f_\Q)\leq N_D$ is definable in $f$ (or indeed in $N_D$ where $N_D$ inherits all the $f$-definable structure it could from $f$). This proves that $f_\Q$ is interpretable in $f$ and so of finite Morley rank. Moreover any complement $N^0$ of the divisible subgroup $im(f)$ of $N_D$ is interpretable in $f$, and so $f^0: C(f)\times C(f) \to N^0$ is interpretable in $f$ and has finite Morley rank too. 

(2) $\Rightarrow$ (1) is straightforward. Just recall that every bilinear map of torsion-divisible abelian groups is identically degenerate (See Equation~\eqref{tor-div:eqn} in the proof of Proposition~\ref{InDec:stablebilin:prop} for instance.).\end{proof}
\begin{cor} Assume $f:M\times M \to N$ is a full non-degenerate bilinear map. Then the following statements are equivalent.
\begin{enumerate}
\item $f$ has finite Morley rank.
\item $f=f_\Q\oplus f_C$ where both $f_\Q$ and $f_C$ have finite Morley rank.  
\end{enumerate}  
\end{cor}
\begin{proof} Proof (1.)$\Rightarrow$ (2.) is similar to the proof of Corollary~\ref{st:fullnodeg:cor}. (2.)$\Rightarrow$ (1.) simply follows from the fact that direct sum of two structures of finitely Morley rank has finite Morley rank.\end{proof}
\begin{prop}[\cite{M90b}, Theorem 5.4] \label{Qbilin:fmr:prop} Assume $f:M_\Q\times M_\Q \to N_\Q$ is a bilinear map of torsion-free divisible abelian groups. Then the following statements are equivalent.
\begin{enumerate}
\item $f$ has finite Morley rank.
\item $f$ decomposes into $f=f_1 \oplus \cdots \oplus f_n \oplus f^0$. Each factor is definable in $f$ and has finite Morley rank. Each $f_i$, $i=1,\ldots n$ is a directly indecomposable $k_i$-bilinear mapping of finite dimensional $k_i$-spaces for some algebraically closed field $k_i$ of characteristic zero and $f^0$ is an identically degenerate $\Q$-bilinear map. 
\end{enumerate}\end{prop}

The following statement is an immediate corollary of propositions~\ref{Div:fmrBilin:prop} and \ref{Qbilin:fmr:prop}.
\begin{cor}\label{Div:fmrBilin:cor} Assume $f:M_D\times M_D\to N_D$ is a bilinear map of divisible abelian groups. Then the following are equivalent:   
\begin{enumerate}
\item $f$ is of finite Morley rank.
\item $f$ decomposes into $f=f_1 \oplus \cdots \oplus f_n \oplus f^0$. Each factor is definable in $f$ and has finite Morley rank. Each $f_i$, $i=1,\ldots n$ is a directly indecomposable $k_i$-bilinear mapping of finite dimensional $k_i$-spaces for some algebraically closed field $k_i$ of characteristic zero and $f^0$ is an identically degenerate bilinear map of divisible abelian groups. 
\end{enumerate}\end{cor}
\begin{defn} Assume $v(x_1, \ldots , x_m)$ is a term (word) of the language of rings. The verbal ideal, $v(\bar{A})$, of a ring $A$ defined by $v(\bar{x})$ is the ideal of $A$ generated by all elements $v(\bar{a})$, $a_1,\ldots a_m\in A$. i.e.
$$\alpha \in v(\bar{A}) \Leftrightarrow \exists n \in \N \quad \alpha=\sum_{i=1}^n v(\ov{a_i}).$$\end{defn} 

For example if $I$ is an ideal of a ring $A$, then $I^2$ is the verbal ideal defined by the term $v(x,y)=x\cdot y$. A verbal ideal $v(\bar{A})$ defined by $v(\bar{x})$ is said to have \emph{finite width} if there is a number $n\in \N$ so that every element $\alpha$ of $v(\bar{A})$ can be written as a sum of no more than $n$ elements of the form $v(\ov{a_i})$, $i=1, \ldots , n$, $a_{i1}\in A, \ldots , a_{im} \in A$.  

Now assume $A$ is a commutative associative unitary Noetherian local ring, whose maximal ideal $J$ is nilpotent, say $J^n=0$, and the field of residues $A/J$ is $k$. Then we get a sequence of ideals 
$$A=J^0 > J > J^2 > \ldots >J^n=0$$
where each quotient $J^i/J^{i+1}$ is a $k$-vector space. The Noetherity condition implies that the dimension $dim_k(J^i/J^{i+1})$ of $J^i/J^{i+1}$ over $k$ is finite.
Define $$r_k(A)=\sum_{i=1}^n dim_k(J^i/J^{i+1}).$$ In general if $M$ is an $A$-module for a ring $A$ as above we define $$r_k(M)=\sum_{i=1}^n dim_k(J^{i-1}M/J^{i}M).$$   

\begin{lem}\label{local-alg:lem} Let $A$ be a local algebra whose unique maximal ideal $J$ is nilpotent and the field $k=A/J$ is algebraically closed. Consider the two-sorted structure $M_A=\langle M, A, s \rangle$, where $s$ is the predicate describing the action of $A$ on $M$. If $M_A$ is of finite Morley rank, then $r_k(M)\leq RM(M)$, where $RM(M)$ is the Morley rank of $M$, $M$ inheriting all the definable structure it possibly could from $M_A$. In particular $A$ is Noetherian and $M$ is a finitely generated module over $A$. \end{lem}

\begin{proof} This is really a restatement of Theorem 3.2 from \cite{CR}. The proof of the referred theorem will go through without any substantial changes.\end{proof}

\begin{prop} \label{bilin:fmrB:mainprop} Assume $f:M_B\times M_B \to N_B$ is a bilinear map of finite Morley rank of bounded abelian groups. Then:
	
\begin{enumerate}
\item The foundation
$$f^F:M_B/C(f)\times M_B/C(f)\to im(f)$$ of $f$ decomposes into  
$$f^F=f^F_1\oplus \cdots \oplus f^F_m.$$
Each $f^F_i$, $i=1, \ldots, m$ is a directly indecomposable $P_i$-bilinear mapping, where each ring $P_i$ is an indecomposable ring of finite Morley rank.
\item The components  $f^F_i$ where $k_i=P_i/J_i$ is algebraically closed are $P_i$-bilinear maps of finitely generated $P_i$-modules where $P_i$ is Noetherian. In particular those components are definable in $f^F$. 
\end{enumerate} 
\end{prop}

\begin{proof}
The first statement is similar to the $\omega$-stable case (Proposition~\ref{boundedstablebilin:prop}). First part of statement (2) follows form (1) and Lemma~\ref{local-alg:lem}. For the components referred in (2), $im(f^F_i)$ is also  finitely generated over the corresponding scalar ring $P_i$ and so it is a bilinear map of finite width. The statement follows now. 
\end{proof}
The following corollary will be useful in our analysis of algebras of finite  Morley rank.

\begin{cor}\label{bilinfmr:fw:cor} Let $f:M\times M\to N$ be a  bilinear map of finite Morley rank, where all the residue fields $k_i=P_i/J_i$ of components $P_i$ of $P(f)$ are algebraically closed. Then the image $im(f)$ of $f$ is definable in $f$. \end{cor}
\begin{proof} It should be clear from Proposition~\ref{Qbilin:fmr:prop} and Proposition~\ref{bilin:fmrB:mainprop} that each component $f_i$ is of finite width. Therefore their direct sum is of finite width. This proves that $f:M\times M \to im(f)$ is a full bilinear map of finite width. This implies that $im(f)$ is definable in $f$.\end{proof} 
 
\section{$\omega$-stable algebras} \label{algebras:sec}
This section contains a rather complete analysis of the structure of $\omega$-stable rings (of finite Morley rank). Again similar to bilinear maps we do not give a single structure theorem but rather a collection of statements. We emphasis that rings (algebras) are not assumed to be associative, commutative or unitary (for example Lie algebras). We use the term ``algebra'' interchangeably with ``ring'' for these arbitrary rings. However the scalar rings are always assumed to be commutative, associative and unitary.

\subsection{Some preliminary facts on algebras}
 Consider an arbitrary ring $R$. Define the \emph{two-sided annihilator ideal (center)} of $R$ by $$Ann(R)=\{x\in R: xy=yx=0, \forall y\in R\}.$$ Let $R^2$ be the ideal of $R$ generated by all products $x\cdot y$ (or $xy$ for short) of elements of $R$. 

\subsubsection{Foundations and additions}

Consider a scalar ring $A$. An $A$-algebra $R$ is called \emph{regular} if $Ann(R)\leq R^2$. A maximal regular $A$-subalgebra $R_F$ of $R$ containing $R^2$ is called a \emph{foundation} of $R$ and an \emph{addition} $R_0$ of $R$ is a direct complement (if it exists) of the $A$-submodule $R^2\cap Ann(R)$ in $Ann(R)$.
\begin{prop}[\cite{M90c}, Proposition 6] \label{fa:prop} For a $k$-algebra $R$, where $k$ is a field, the following hold.
\begin{enumerate}
\item One can always construct a foundation $R_F$ for $R$.
\item A proper subalgebra $R_F$ and a submodule $R_0\leq Ann(R)$ are a foundation and an addition of $R$ respectively if and only if $R=R_F\times R_0$.
\item Different foundations of $R$ are pairwise isomorphic, and so are different additions of $R$. 
\end{enumerate}
\end{prop}

In general if $R$ is not a $k$-algebra over a field $k$ the situation is not as nice as above. In particular an addition $R_0$ may not exist. Even if an addition exists it may not split from $R$. We provide some examples below. 

\begin{exmp} \label{addition:exmp} In the following $(~,~)$ denotes the Lie bracket. So consider the nilpotent Lie ring $R$ with presentation:
$$R=\langle s, t, u : (s,u)=(t,u)=0, (s,t)=2u \rangle.$$
It is not hard to see that $Ann(R)=\langle u\rangle\cong \Z$, while $R^2\cap Ann(R)=\langle 2u\rangle\cong \Z$ does not split from $Ann(R)$, and therefore $R_0$ does not exist.

Next consider the Lie ring $S$ with the following presentation in the variety of nilpotent Lie rings of class 2:
$$S=\langle s,t,u,v: (s,v)=(t,v)=(u,v)=0, (3u,s)=(3u,t)=0, (s,t)=v\rangle.$$
Again it is not hard to see that $Ann(S)=\langle 3u,v\rangle$, $S^2=\langle v\rangle$, and $S_0=\langle3u\rangle$. Clearly $S$ does not split over $S_0$.  
\end{exmp}

\subsubsection{Largest ring of scalars $A(R)$} Let $R$ be an $A$-algebra where $A$ is a scalar ring. Here we only consider those algebras which are faithful with respect to the action of their rings of scalars. Let $\mu: A\to A_1$ be an inclusion of rings. We say that an $A$-algebra $R$ has an $A_1$-enrichment with respect to $\mu$ if $R$ is an $A_1$-algebra and $ \alpha r =\mu(\alpha)r$, $r\in R$, $\alpha \in A$.  
 
Assume $R$ is an $A$-algebra, where $A$ is an associative commutative unitary ring. Denote by $A(R)$ the largest, in the sense defined just above, commutative subring of $End_A(R/Ann(R))$ that satisfies the following conditions:
\begin{enumerate}
\item $R/Ann(R)$ and $R^2$ are faithful $A(R)$-modules.
\item The full non-degenerate bilinear mapping $f_F: R/Ann(R)\times R/Ann(R)\to R^2$ induced by the product in $R$ is $A(R)$-bilinear.
\item The canonical homomorphism $\eta:R^2 \to R/Ann(R)$ is $A(R)$-linear.
\end{enumerate}

\begin{prop}[\cite{M90c}, Proposition 8]\label{A(R)exists:prop} For any algebra $R$ the ring $A(R)$ is definable, it is unique, and does not depend on the choice of the initial ring of scalars.\end{prop}

\begin{prop}[\cite{M90c}, Proposition 9]\label{decomringalg:prop} Let $k$ be a field and $R$ a regular $k$-algebra. Then 
$R\cong R_1\times \cdots \times R_n$ for $k$-subalgebras $R_i$ if and only if $A(R)\cong A_1\times \ldots \times A_n$, where $A_i\cong A(R_i)$, $i=1,\ldots, n$.
\end{prop}

\begin{prop}[\cite{M90c}, Proposition 10]\label{L-enrich:prop} Let $R$ be a $k$-algebra for a field $k$. Then any foundation of $R$ is an $L$-algebra for any ring $L$, $k\leq L \leq A(R)$.\end{prop}

\subsection{$\omega$-stable algebras: The characteristic zero case}\label{wstable-al:sec}

Again we remind that a ring $R$ has {\em characteristic zero} if its additive group $R^+$ is torsion-free. 

\begin{lem}\label{A(R):int:lem} Let $R$ be an $\omega$-stable ring of characteristic zero. Then:
\begin{itemize}
\item [1)]  $R$ is a $\mathbb{Q}$-algebra. 
\item [2)] The  two sorted structure $\langle A(R), R, s\rangle$ where $s$ describes the action of $A(R)$ on $R$ is absolutely interpretable in the pure ring $R$.
\end{itemize}
\end{lem}
\begin{proof} Firstly note that $R$ is a $\Q$-algebra by Theorem~\ref{Mac:stable:tfa}. Therefore the product in $R$ induces a $\Q$-bilinear mapping 

$$f:R \times R \to R, \qquad f(x,y)=xy, \quad \forall x,y \in R.$$

Let $f': R/Ann(R)\times R/Ann(R)\to R$, be the non-degenerate bilinear map induced by $f$. $f$ is absolutely interpretable in $A$, hence is $f'$. So $f'$ is $\omega$-stable. Recall the construction of $Z(f')$ (Equation~\eqref{Z(f)}) and its subrings $Z_n(f')$ (Equation~\eqref{zn}) associated to $f'$ from Section~\ref{bilin-scalar:sec}. By proof of (\cite{M90b}, Proposition 5.1) $Z_n(f')\leq Z(f')\leq Sym_{f'}(R/Ann(R))$ are all absolutely definable in $f'$ for every $n$ since $f'$ is $\omega$-stable. Indeed all these rings are interpretable in $R$. Now let $$^-: R \to R/ Ann(R)$$ be the canonical projection. By definition each element $\alpha\in Z_n(f')$ has a well defined action on any element $\ds z=\sum_{i=1}^nf'(\ov{x_i},\ov{y_i})=\sum_{i=1}^nf(x_i,y_i)$ of $R^2$ by setting 
$$\alpha\cdot z=\sum_{i=1}^nf'(\alpha\ov{x_i},\ov{y_i}).$$
Now, for each $n$ consider the sub-algebra $L_n$ of $Z_n(f')$ consisting of those $\alpha$ in $Z_n(f')$ where
$$\alpha\left(\overline{\sum_{i=1}^nf(x_i,y_i)}\right)=\overline{\left(\alpha\sum_{i=1}^nf(x_i,y_i)\right)}.$$
For each $n$, $L_n$ is definable in $Z_n(f')$, hence $L_n$ is definable in $Z(f')$. Note that 
$$A(R)= \cap_{i=1}^\infty L_n.$$ 
As $Z(f')$ is $\omega$-stable as a group interpretable in $R$ and the $L_n$ are definable in $Z(f')$, by the descending chain condition  
$$A(R)= \cap_{i=1}^\infty L_n=\cap_{i=1}^tL_t=L_t$$
for some finite $t$. This proves that $A(R)$ which is commutative associative unitary ring is interpretable in $R$.\end{proof}

\noindent \emph{Proof of Theorem~\ref{w-st-al:mainthm}.} 
By Lemma \ref{A(R):int:lem} the commutative associate unitary ring $A(R)$ is interpretable in $R$. Therefore $A(R)$ is an $\omega$-stable commutative associative unitary ring of characteristic zero and therefore by Theorem~\ref{CR:stable} it admits a decomposition 
$$A(R)\cong A_1\times \cdots \times A_n$$
where each $A_i$ is a local ring whose maximal ideal $J_i$ is nilpotent and the field of residues $k_i=A_i/J_i$ is a characteristic zero algebraically closed field. By Proposition~\ref{decomringalg:prop}, any foundation $R_F$ of $R$ admits a decomposition
$$R_F=R_1\times \cdots \times R_n$$ 
with $A(R_i)=A_i$. Since $R$ is a $\Q$-algebra we have $R=R_F\times R_0$ where $R_0$ is an addition of $R$ which is a $\Q$-algebra with zero multiplication.

As noted before, $A_i$ contains copies of $k_i=A_i/J_i$ called fields of representatives of $k_i$ and so by Proposition~\ref{L-enrich:prop} each $R_i$ is a $k_i$-algebra. 

\qed
\subsection{$\omega$-stable algebras: The general case}
We call a ring $R$ an \emph{A-quasi algebra} for a scalar ring $A$ if $R/Ann(R)$ and $R^2$ are $A$-algebras and the canonical homomorphism $\eta: R^2 \to R/Ann(R)$ is $A$-linear. We need to introduce this notion since in the general case additions and therefore foundations may not exist. Even when they exit there are examples where $R/Ann(R)$ and $R^2$ carry compatible $A$-algebra structures but $R$ itself does not admit any $A$-algebra structure compatible with those of $R/Ann(R)$ and $R^2$ (See~\cite{S}, Proposition~3.2.5).   

We say that $R$ is a \emph{central product} of subrings $R_1, \ldots ,R_n$ if $R=R_1+\cdots +R_n$, and $R_iR_j=R_jR_i=0$, if $i\neq j$. We denote a central product by $R=R_1\ast\cdots \ast R_n$. Clearly in a central product as above $R_i\cap R_j\leq Ann(R)$ if $i\neq j$.

We say that a ring $R$ is a \emph{divisible ring} if its additive group $R^+$ is a divisible abelian group. Similarly we say that $R$ is \emph{bounded} if it is additive group $R^+$ has  finite exponent. 
\begin{prop}\label{wstable:aldecom:prop} An $\omega$-stable ring $R$ decomposes into $R=R_D \ast R_C$ where both $R_D$ and $R_C$ are absolutely definable in $R$, $R_D$ is a divisible subring and $R_C$ is a bounded subring.\end{prop}
\begin{proof} By Macintyre's Theorem \ref{Mac:stable:tfa}, there are subgroups $R_D$ and $R_C$ of the additive group $R^+$ of $R$ with the given properties. We recall that $R^+_D=mR^+$ and $R_C^+=\{x\in R: mx=0\}$ where $m$ is the exponent of $R_C^+$. This also clearly shows that both $R_D$ and $R_C$ are closed under the ring multiplication and so both are subrings of $R$. Also if $x\in R_D$ and $y\in R_C$ there exists $x'\in R_D$ such that $mx'=x$. Therefore
	$xy=x'(my)=0=(my)x'=yx$, which means they annihilate each other and clearly $R_D\cap R_C\leq Ann(R)$. Both $R_D$ and $R_C$ are $\omega$-stable since both are definable in $R$. \end{proof}
\begin{prop}\label{decom:wstable:prop} Assume $R$ is a divisible $\omega$-stable ring. Then $Ann(R)$ is divisible and both $R/Ann(R)$ and $R^2$ are torsion-free divisible. Moreover there exists an addition $R_0$ of $R$ and a corresponding foundation $R_F$ such that $R=R_F \times R_0$ and $R_F$ is a $\Q$-algebra.\end{prop}
\begin{proof}The proof of the first part is completely similar to the first part of the proof of Proposition~\ref{Div:stableBilin:prop} from bilinear maps. So we omit it and proceed to the proof of ``Moreover" part.

Firstly we show that $R^2\cap Ann(R)$ splits from $Ann(R)$. Since in $Ann(R)$ the multiplication is trivial it is enough to prove that $Ann(R)$ splits over $R^2\cap Ann(R)$ as an abelian group. For that purpose we show that $\Delta (R)=Ann(R)\cap R^2$ is divisible. Assume $z=\sum x_iy_i\in \Delta(R)$. Given any arbitrary non zero integer $m$ the exists $x'_i$ such that $mx'_i=x_i$. We claim $z'=\sum x'_iy_i \in \Delta(R)$ while clearly $mz'=z$. All we need to show is that $z' \in Ann(R)$ which was established in a similar situation in bilinear maps (See the proof of Proposition~\ref{Div:stableBilin:prop}). So $\Delta(R)$ is a divisible group. Therefore $Ann(R)$ splits over $\Delta(R)=Ann(R)\cap R^2$. 

So let $R_0$ be an addition of $R$, i.e. a subgroup of $Ann(R)$ such that $R_0\oplus \Delta(R)=Ann(R)$. we claim that $R_0$ is divisible. Indeed pick $x\in R_0$ and any non-zero integer $m$ and choose $x'\in Ann(R)$ such that $mx'=x$. Therefore there are unique $x'_1\in R_0$ and $x'_2\in \Delta(R)$ such that $x'=x_1'+x_2'$. So $mx_1'+mx_2'=x$. This shows that $mx'_2\in \Delta(R)\cap R_0=0$. Since $\Delta(R)\leq R^2$ is torsion-free we have proved that $x_2'=0$ and so $x'\in R_0$. On the other hand since $R^2$ is divisible there is a subgroup $C$ of $R^+$ such that $R^+=C\oplus R^2$. Since $R^+$ s divisible, $C\cong R^+/R^2$ is divisible too. 

Assume now that $T$ denotes the maximal torsion subgroup of $R^+$. Next we show that $T\leq Ann(R)$. Pick $x\in T$ and $y\in R$. Since $x$ is a torsion element there is a non-zero integer $m$ such that $mx=0$. Let $y'$ be an element such that $my'=y$. Then 
$$xy=x(my')=(mx)y'=0=y'(mx)=(my')x=yx$$
implying that $x\in Ann(R)$. Since $R^2$ is torsion-free we need to have $T\leq R_0$ for any choice of $R_0$. Since $R_0$ is divisible and $T$ is its maximal torsion subgroup there is a divisible torsion-free subgroup $S$ of $R_0$ such that $R_0=S\oplus T$. Finally by the same argument there is a divisible torsion-free subgroup $Q$ of $C$ such that $C=Q\oplus R_0=Q\oplus S \oplus T$. Now define $R_F$ as the subring of $R$ generated by $Q$ and $R^2$. The ring $R_F$ is clearly a foundation of $R$ and a $\Q$-algebra as an extension of the $\Q$-vector space $R^2$ by the $\Q$-vector space $Q$.
 \end{proof} 
  
\begin{prop}\label{wstable:boundedalg:prop} Assume $R$ is an $\omega$-stable bounded ring. Then there are local rings $A_i$ with maximal ideal $J_i$ such that the field $F_i=A_i/J_i$ is either algebraically closed of positive characteristic or finite and $R$ is a central product 
	$$R=R_1\ast\cdots \ast R_n$$ 
of $A_i$-quasi algebras $R_i$. \end{prop}
\begin{proof}
Similar to the proof of Lemma~\ref{A(R):int:lem} the action of scalar ring $A(R)$ on $R/Ann(R)$ is interpretable in $R$. Hence $A=A(R)$ is $\omega$-stable. So $A$ has a Cherlin-Reineke decomposition $A=A_1\times \cdots \times A_n$ as in Theorem~\ref{CR:stable}. In this case since $R^+$ has bounded exponent the $F^+_i$ will have bounded exponent so the $F_i$ have positive characteristic. By definition $A$ has well-defined action on $R/Ann(R)$ and $R^2$ and the canonical map $\eta:R^2 \to R/Ann(R)$ is $A$-linear. So indeed $R$ is an $A$-quasi algebra. Now consider  $\bar{R}_i=A_i(R/Ann(R))$ and let $R_i$ be inverse image of $\bar{R}_i$ under the canonical map $\bar{~}:R\to R/Ann(R)$. Clearly $R$ is a central product of the $R_i$. Moreover if $x\in Ann(R_i)$ then $x\in Ann(R)$ since $R_i$ annihilates every $R_j$, $i\neq j$. So $Ann(R_i)=Ann(R)$. Therefore by construction $R_i/Ann(R_i)$ and $R_i^2$ are $A_i$-algebras and the canonical map $\eta_i: R_i^2\to R_i/Ann(R_i)$ is $A_i$-linear. so each $R_i$ is an $A_i$-quasi algebra. This finishes the proof. 
\end{proof}
\subsection{Algebras of finite Morley rank}\label{fmr-alg:sec}
\begin{prop} Assume $R$ is a ring. Then $R$ has finite Morley rank if and only if $R=R_D\ast R_C$ where $R_D$ and $R_C$ are both absolutely definable in $R$, $R_D$ is a divisible ring of finite Morley rank, $R_C$ is a bounded ring of finite Morley rank and $|R_D\cap R_C|<\infty$.\end{prop}
\begin{proof} Most of the proof of the only if direction is included in the proof of Proposition~\ref{wstable:aldecom:prop}. We just need to check $|R_D\cap R_C|<\infty$. By definition $R_C=mR^+$ where $m$ is the exponent of $R_C$. Since $R_D^+$ is divisible abelian of finite Morley rank it follows from Theorem~\ref{Mac:stable:tfa} that $R_D^+= (\bigoplus_I\Q^+) \oplus (\bigoplus_{p}(\bigoplus_{I_p} \Z(p^\infty)))$ where $p$ ranges over all primes and for each $p$ the index set $I_p$ is finite. This means the set of elements of exponent $m$ in $R_D$ is finite. In particular $|R_D\cap R_C|<\infty$. For the other direction note that $R_D\ast R_C\cong (R_D\times R_C)/I$ where $I=\{(a,-a):a\in R_D\cap R_C\}$. The ideal $I$ is definable in the product $R_D\times R_C$ since $|R_D\cap R_C|<\infty$. The product $R_D\times R_C$ is of finite Morley rank, so is $R_D \ast R_C$, since it is interpretable in $R_D\times R_C$. \end{proof}

\begin{lem}\label{R2:defn:lem} Let $R$ be a divisible ring of finite Morley rank. Then the ideal $R^2$ is definable in $R$.\end{lem}

\begin{proof} The natural non-degenerate bilinear map $f:R/Ann(R)\times R/Ann(R) \to R $ associated to $R$ is interpretable in $R$ so it is a bilinear map of finite Morley rank. By Proposition~\ref{decom:wstable:prop} $R/Ann(R)$ is torsion-free. Note that $R^2=im(f)$ where the later is definable in $f$ since $f$ satisfies hypothesis of Corollary~\ref{bilinfmr:fw:cor} now. Therefore $R^2$ is definable in $R$.  
 \end{proof} 
 
\noindent \emph{Proof of Theorem~\ref{fmr-alg:mainthm}.} Most of the proof of the only if direction is already included in the proof of Theorem~\ref{w-st-al:mainthm}. However in this case $A(R)$ satisfies the hypothesis of Theorem~\ref{CR:fmr}. Therefore $A(R)$ is decomposed into a finite direct sum 
$$A(R)=A_1\times\cdots \times A_n$$ of Noetherian local rings $A_i$, where the maximal ideal of $A_i$ is nilpotent and the field of residues is a characteristic zero algebraically closed field. Similar to proof of Theorem~\ref{w-st-al:mainthm} we find that $R$ can be decomposed into a finite direct sum $$R\cong R_1\times\cdots \times R_n\times R_0$$ where each $R_i$, $i\neq 0$, is an indecomposable $k_i$-algebra, $k_i$ is a characteristic zero algebraically closed field, and $R_0$ is a central (and in particular with zero multiplication) $\Q$-subalgebra of $R$.

The only part left in this direction is to prove that $R_i$ is a finite dimensional $k_i$-algebra, $i=1, \ldots , n$. Firstly $A_i$ is a Noetherian $k_i$-algebra and therefore a finite-dimensional $k_i$-algebra.  Since $R_i=e_iR_F$, where $e_i$ is the unit of $A_i$, is a central extension of $e_i(Ann(R)\cap R^2)$ by $e_i(R/Ann(R))$ it is enough to prove that each of these $k_i$-vector spaces are finite dimensional over $k_i$. Note that $\bar{R}_A=_{\text{def}}\langle R/Ann(R), A(R), s\rangle$, where $s$ is the predicate describing the action of $A(R)$ on $R/Ann(R)$, is interpretable in $R$ so it has finite Morley rank. Furthermore $\langle e_i(R/Ann(R)), A_i, s_i\rangle$ is definable in $\bar{R}_A$ using the constant $e_i$, so the former is also of finite Morley rank. So by Lemma~\ref{local-alg:lem}, $e_i(R/Ann(R))$ is a finite-dimensional vector space over $k_i$. For $e_i(R^2\cap Ann(R))$ the argument is similar but we need to use Lemma~\ref{R2:defn:lem} to show that $R^2\cap Ann(R)$ is definable in $R$.
     
Let us prove the sufficiency of the conditions above. It is enough to show that each direct summand has finite Morley rank. Each of the algebras $R_i$ is interpretable in the corresponding algebraically closed field $k_i$ using some fixed structural constants. So each $R_i$ has finite Morley rank. The $\Q$-algebra $R_0$ is assumed to have zero multiplication so its definable structure is no richer than that of a vector space over the field of rationals $\Q$, so it is of finite Morley rank. This finishes the proof.
\qed

\begin{prop} For a divisible ring $R$ the following statements are equivalent.
	\begin{enumerate}
		\item $R$ has finite Morley rank
		\item $R$ admits the decomposition $R=R_F \times R_0$ where $R_F$ is a characteristic zero regular ring of finite Morley rank and $R_0$ is divisible algebra of finite Morley rank with zero multiplication.
	\end{enumerate}\end{prop}
	\begin{proof} Statement (2) clearly implies (1). Most of the other direction is already included in proof of Proposition~\ref{decom:wstable:prop}. Remains to prove the statements regarding Morley rank of the components. The addition $R_0$ is isomorphic to $Ann(R)/(Ann(R)\cap R^2)$. The later is interpretable in $R$ since both $Ann(R)$ and $R^2$ are definable in $R$. So $R_0$ has finite Morley rank. Similar to the proof of Theorem~\ref{fmr-alg:mainthm} the scalar ring $A(R)$ has characteristic zero since $R/Ann(R)$ is torsion-free. The ring $A(R)$ is interpretable in $R$, hence has finite Morley rank and makes $R_F$ an $A(R)$-algebra. Again by a similar argument as in the proof of Theorem~\ref{fmr-alg:mainthm}, $$R_F=R_1\times \cdots \times R_n$$
	where each component is a finite-dimensional $k_i$-algebra for an algebraically closed field $k_i$ of characteristic zero. So indeed each $R_i$ is interpretable in a $k_i$ with the help of a finite number of fixed structure constants. Hence each $R_i$ has finite Morley rank. Therefore $R_F$ has finite Morley rank.      \end{proof} 

\begin{prop}\label{fmr:boundedalg:prop} Assume $R$ is a bounded ring of finite Morley rank. Then there are local rings $A_i$ with maximal ideal $J_i$ such that the field $F_i=A_i/J_i$ is either algebraically closed of positive characteristic or finite and $R$ is a central product 
	$$R=R_1\ast\cdots \ast R_n$$ 
	of $A_i$-quasi algebras $R_i$. If $F_i$ is algebraically closed then $A_i$ is Noetherian and both $R_i/Ann(R_i)$ and $R_i^2$ are finitely generated $A_i$-modules.\end{prop}
\begin{proof} Again most of the statement is a corollary of Proposition~\ref{wstable:boundedalg:prop}. The statement regarding Noetherity of the $A_i$ and finite generation of $A_i$-module  $R_i/Ann(R_i)$ where $F_i$ is algebraically closed follows from Lemma~\ref{local-alg:lem}. Subsequently the statement regarding $R_i^2$ follows. In addition for these $R_i$, $R_i^2$ is of finite width and definable. If $F_i$ is just finite no such statement can be made.\end{proof}  
\subsection{Uncountably categorical algebras}

We can now prove a criterion for uncountable categoricity of a characteristic zero algebra. We begin by discussing some known facts and fixing some notation. So assume that $R$ is  a directly indecomposable finite-dimensional $k$-algebras without zero multiplication.  Now the ring $A(R)$ is a finite dimensional $k$-algebra and we are in the set up of Lemma~\ref{local-alg:lem} and its proof. Let's call the series 
$$R> JR > \ldots >J^mR=0,$$
a \emph{$J$-series} for $R$ and a $k$-basis $\bar{u}=(u_1, \ldots , u_d)$ adapted to this series a \emph{special} $k$-basis for $R$, and set $k(R)=_{\text{def}}A/J$. Pick a field of representatives $L$ of $k(R)$ in $A(R)$ and consider the subfield $k_0$ of $L$ generated by the structure constants associated to $\bar{u}$. Set $U_0$ to be the $k_0$ hull of $\bar{u}$ in $R$. Then:

\begin{thm}[\cite{M90c}, Theorem 8]\label{models-alg:thm} Any model $S$ of the complete theory of the directly-indecomposable $k$-algebra $R$ has the form $U_0\otimes_{k_0}k(S)$ where $k(S)\equiv k(R)$.\end{thm}

\begin{thm}\label{cat:alg:thm} Assume $R$ is a characteristic zero ring without zero multiplication. Then $R$ is $\aleph_1$-categorical if and only if it is a directly indecomposable $k$-algebra without zero multiplication, where $k$ is an uncountable characteristic zero algebraically closed field.\end{thm}
\begin{proof} The proof is standard but we provide some detail here. For the only if direction note that uncountable categoricity implies finite Morley rank. So $R$ admits a decomposition as in Theorem ~\ref{fmr-alg:mainthm}. Now Assume $R$ has uncountable cardinality say $\lambda$. So at least one component $R_i$ for some $1 \leq i\leq n$ or $R_0$ has to have cardinality $\lambda$. If there is more than one component other than the one we already picked one can replace it by a model of cardinality $\mu < \lambda$ contradicting the categoricity assumption. Since the algebra is assumed to be without zero multiplication $R$ is equal to $R_i$, $i\neq 0$. 

For the if direction by Theorem~\ref{models-alg:thm}, $R\cong U_0\otimes_{k_0}k(R)$. The field $k(R)$ is a finite extension of the algebraically closed field $k$ so indeed $k=k(R)$. Now Assume $R$ has uncountable cardinality $\lambda$. If $S\equiv R$ and they have the same cardinality $\lambda$ then $k(S)\equiv k(R)\equiv k$ and they have to have cardinality $\lambda$ since $k_0$ is a countable field by construction. The result follows from uncountable categoricity of algebraically closed fields. \end{proof}

\begin{rem} As mentioned in the introduction this theorem is due to B. Zilber~\cite{Z82} for nilpotent Lie algebras. Our result seems to be new in this generality to the best of our knowledge.\end{rem}  

\section{Bi-interpretability of $k$-groups and nilpotent Lie $k$-algebras}\label{malcev-cor:sec}
\subsection{Mal'cev correspondence between $k$-groups and nilpotent Lie $k$-algebras}

Here we briefly discuss the notion of a nilpotent group admitting exponents in a characteristic zero field $k$ or a $k$-group for short. The details may be found in either of \cite{hall, war}.

\begin{defn}\label{kgroups:defn}
A group $G$ admitting exponents in a characteristic zero field $k$ or a $k$-group for short is a nilpotent group $G$ together with a function:
$$G\times k\rightarrow G, \quad (x,a)\mapsto x^a,$$
satisfying the following axioms:
\begin{enumerate}
 \item $x^1=x$, $x^a x^b=x^{(a+b)}$, $(x^a)^b=x^{(a b)}$, for all $x\in G$ and $a, b\in R$.
 \item $(y^{-1}xy)^{a}=y^{-1}x^a y$ for all $x,y\in G$ and $a\in R$.
 \item  $x_1^a x_2^a\cdots x_n^a=(x_1x_2\cdots x_n)^a\tau_2(\bar{x})^{\binom{a}{2}}\cdots \tau_c(\bar{x})^{\binom{a}{c}}$, for all $x_1$, \ldots, $x_n$ in $G$, $a \in R$, where $\tau_i$ come from Hall-Petresco formula and $c$ is the nilpotency class of $G$.
\end{enumerate}\end{defn}
Now we can define a homomorphism of $k$-groups.
\begin{defn}Let $G$ and $H$ be two $k$-groups. A map $\phi:G\rightarrow H$ is a homomorphism of $k$-groups if
\begin{itemize}
 \item $\phi$ is a homomorphism of groups,
 \item $\phi(x^a)=(\phi(x))^a$ for all $x\in G$ and $a\in k$.
\end{itemize}
\end{defn}

The following theorem was proved first in the case where $k=\Q, \R$ by A.I. Mal'cev in the finite dimensional case. Later it was noticed by M. Lazard (in an unpublished paper) that finite-dimensionality condition can be removed. D. Quillen~\cite{Q} extended these results to the pro-nilpotent case using the theory of complete Hopf-algebras. Finally R. B. Warfield~\cite{war} sketched, using Quillen's approach, a generalization to the case of arbitrary $k$-groups, where $k$ is a field of characteristic zero.

Just a note regarding notation. We will write $L_1\oplus L_2$ for the direct sum (or product) of Lie algebras $L_1$ and $L_2$ and we reserve $G_1\times G_2$ for direct product of groups $G_1$ and $G_2$.    
\begin{thm}\label{malcevcorrespondence:thm} Let $k$ be a field of characteristic zero, $\mathcal{G}$ be the category of $k$-groups and $\mathcal{L}$ be the category of nilpotent Lie $k$-algebras. Then there is a category equivalence
$$\rho: \mathcal{L} \to \mathcal{G},$$
such that for any $L\in \mathcal{L}$ and $G=\rho(L)$, there are maps
$$\exp: L \to G, \quad \log: G \to L,$$ satisfying
\begin{enumerate}
\item $\exp(\log(x))=x$ for any $x\in \rho(L)$ and  $\log(\exp(x)) =x$ for any $x\in L$,
\item if $(x,y)$ denotes the Lie bracket in $L$ then $\exp$ satisfies the Campbell-Baker-Hausdorff formula:
$$\exp~ x \cdot \exp~ y =\exp ( x+y +\frac{1}{2}(x,y) + \cdots) $$
where $\cdots $ is a linear combination of Lie brackets of weight greater than 2 with rational coefficients,
\item moreover 
$$[g_1,g_2, \ldots , g_n]=\exp((l_1,\ldots , l_n)+\cdots )$$  
where $g_i=\exp(l_i)$ and $\cdots$ is a finite linear combination of Lie brackets of weight greater than $n$ with rational coefficients, and
$$[x_1,\ldots , x_n]=1 \Leftrightarrow (l_1, \ldots , l_n)=0.$$
\end{enumerate}  
 \end{thm}
 \begin{proof}
 Statement (1) is Theorem 12.11 from \cite{war}, Section 12. For (2) we refer to the first paragraph of \cite{war}, Section 12 again. Statement (3) appears in \cite{baumslag} last line of Page 49 (the discussion following Theorem 4.4). Even though in \cite{baumslag} the statement is for the case $k=\Q$ the whole argument carries over to any characteristic zero field $k$ as far as a suitable notion of a $k$-group is available (We refer to the first few lines of \cite{war}, Section 12). 
 \end{proof}
We collected a few more or less direct consequences of the theorem in the following lemma.   
\begin{lem}\label{bi-inter:lem} Assume $G$ is a $k$-group and $L=log(G)$ is the corresponding Lie algebra. Then the following hold.
\begin{enumerate}
\item $L$ is interpretable in $G$ as a pure group and $G$ is interpretable in $L$ as a pure (Lie) ring. Moreover 
\item $G^i=\exp(L^i)$ where $G^i$ (resp. $L^i$) denotes the $i$-th term of the lower central series of $G$ (resp. $L$).
\item $exp(Ann(L))= Z(G)$
\item The group $G$ is directly decomposable into $G=G_1\times G_2$ if and only if $L=L_1\oplus L_2$ where $L_i=\log(G_i)$. 
\end{enumerate}

\end{lem}
\begin{proof} All the statements are well known. For proofs of (1)-(3) refer to Lemma 2.3 in~\cite{AW}. The proofs in the said reference rely on results from~\cite{stewart}. Using (2) and (3) of Theorem~\ref{malcevcorrespondence:thm} one can check that subgroups and subalgebras correspond to each other. So to prove (4) assume $G=G_1\times G_2$. Then $(L_1,L_2)=0$ by (3) of Theorem~\ref{malcevcorrespondence:thm} since $[G_1,G_2]=1$. The other direction is checked similarly.\end{proof} 
\section{Nilpotent groups}\label{groups:sec}
\subsection{$\omega$-stable nilpotent groups}
In this section we prove Theorems~\ref{w-stable-group:mainthm} and \ref{wstable:arbitgroup:mainthm}. The proofs are really just easy applications of the results obtained in the previous sections.

\begin{lem}\label{stable-divisible} A torsion-free $\omega$-stable nilpotent group is a $\Q$-group.\end{lem}
\begin{proof} Note that for any element $g\in G$ the centralizer $C_G(g)=\{x\in G: xg=gx\}$ is definable in $G$. Then the center $Z(C_G(g))$ of $C_G(g)$ is a non-empty torsion-free abelian group which is definable in $G$. So $Z(C_g(g))$ is a torsion-free $\omega$-stable abelian group, which by Theorem~\ref{Mac:stable:tfa} and structure theory of divisible abelian groups is a $\Q$-vector space. In particular all rational powers of $g$ exist and are uniquely defined. Now the fact that $G$ satisfies axioms (1)-(3) of Definition~\ref{kgroups:defn} follows easily from the existence and uniqueness of roots and the fact that torsion-free nilpotent groups are $\Z$-groups (See~\cite{war}, Section 6) and in particular satisfy Hall-Petresco formula.
 
\end{proof}
Let $G$ be a $k$-group. Then any direct complement $G_0$ of $G'\cap Z(G)$ in $Z(G)$ is called an \emph{addition} of $G$. Since $G$ is a $k$-group then one can actually write $G=G_f\times G_0$ for some $k$-subgroup $G_f\cong G/G_0$ of $G$, called a $\emph{foundation}$ of $G$.\\
 
\noindent\emph{Proof of Theroem~\ref{w-stable-group:mainthm}.} By Lemma~\ref{stable-divisible}, $G$ is a $\Q$-group. By Lemma~\ref{bi-inter:lem}, the Lie $\Q$-algebra $L=\log(G)$ is interpretable in $G$. Hence $L$ is $\omega$-stable. By Theorem~\ref{w-st-al:mainthm}, $L$ can be decomposed into a finite direct sum
$$L\cong L_1\oplus\cdots \oplus L_n\oplus L_0$$ where each $L_i$, $i \neq 0$, is an indecomposable $k_i$-algebra, $k_i$ is a characteristic zero algebraically closed field, and $L_0$ is an addition of $L$, which is a $\Q$-algebra with zero multiplication. By parts (2), (3) and (4) of Lemma~\ref{bi-inter:lem}, $G$ can be decomposed into a finite direct product:
$$G\cong  G_1\times \cdots \times G_n\times G_0,$$
where $G_i=exp(L_i)$, $i=1,\ldots, n$, and $G_0=\exp(L_0)$. Since $L_i$ carries a $k_i$ structure, $G_i$ is a $k_i$-group and $G_0$ is an addition of $G$ and a $\Q$-vector space. 
 
 \qed

\noindent\emph{Proof of Theorem~\ref{wstable:arbitgroup:mainthm}.} The statement is a direct corollary of Theorems~\ref{w-stable-group:mainthm}, \ref{nesin1} and \ref{nesin2}. We would just like to remark that now the addition $D_0$ is a direct product of an abelian $\Q$-group and a divisible torsion abelian group (which is a direct product of all Pr\"ufer $p$-groups, $Z(p^\infty)$).

\qed
\subsection{Nilpotent groups of finite  Morley rank}\label{fmrG:section}~
Here we provide proofs of Theorems~\ref{fmr:groups:mainthm}, \ref{fmr:garbit:mainthm} and \ref{BN-conjecture}. Again the hard work is already done.

\noindent \emph{Proof of Theorem~\ref{fmr:groups:mainthm}.} The if direction is clear. The proof of the only if direction follows the same plan as the proof of Theorem~\ref{w-stable-group:mainthm}, we just need to use the only if part of Theorem~\ref{fmr-alg:mainthm} instead. 

\qed

\noindent \emph{Proof of Theorem~\ref{fmr:garbit:mainthm}.} (1) $\Rightarrow$ (2) is similar to that of Theorem~\ref{wstable:arbitgroup:mainthm}, just use Theorem~\ref{fmr:groups:mainthm} and Theorem~\ref{nesin1a} instead of Theorem~\ref{w-stable-group:mainthm} and Theorem~\ref{nesin1}. Note that the addition $D_0$ is a divisible abelian group so it contains an abelian $\Q$-subgroup $Q$ so that $D_0 = T \times Q$, where $T$ is the torsion part of $D$. The subgroup $T$ has finite Morley rank by Theorem~\ref{nesin2}. The subgroup $Q$ is just a $\Q$-vector space so of finite Morley rank. (2) $\Rightarrow$ (1) follows from Theorem~\ref{nesin1a}. 

\qed

\noindent \emph{Proof of Corollary~\ref{BN-conjecture}.} Just note that the subgroup $N$ coincides with a product $$N=D_1 \times \cdots \times D_n\times Q$$ where each $D_i$ is a unipotent algebraic group over an algebraically closed field of characteristic zero and $Q$ is an abelian $\Q$-group. Since each component of $N$ has finite    Morley rank, hence does $N$.

\qed

The following two theorems follow as corollaries of our results. For Theorem~\ref{AW:thm} just pick the algebraically closed field $k$ whose cardinality is maximal among those of the $k_i$ and $G_0$, from the decomposition given in Theorem~\ref{fmr:groups:mainthm}. The addition $G_0$ will have a 1-dimensional representation over $k$, while the unipotent components have unitriangular representations over $k$ by Theorem 11.7 of~\cite{war}, or by Ado-Iwaswa Theorem and the Mal'cev correspondence. Theorem~\ref{Z:thm} follows from Theorem~\ref{cat:alg:thm} and the Mal'cev correspondence.
  
\begin{thm}[\cite{AW}, Theorem~1]\label{AW:thm} A torsion-free  nilpotent group of finite    Morley rank has a faithful finite-dimensional representation over a field of characteristic zero.\end{thm} 

\begin{thm}[\cite{Z82}] \label{Z:thm} A torsion-free non-abelian nilpotent group is $\aleph_1$-categorical if and only if it is an indecomposable finite dimensional $k$-group, where $k$ is a characteristic zero algebraically closed field.\end{thm}

\end{document}